\documentclass[english,12pt,oneside]{amsproc}
\usepackage[english]{babel}
\usepackage{a4wide}
\usepackage{amsthm}
\usepackage{graphics}
\usepackage{amsfonts, amssymb, amscd, amsmath}
\usepackage{mathdots}
\usepackage{latexsym}
\usepackage[matrix,arrow,curve]{xy}
\usepackage{mathabx}
\usepackage{color}
\usepackage{pbox}
\usepackage{tikz}
\usetikzlibrary{matrix,decorations.pathreplacing,positioning}
\usepackage{hyperref}

\DeclareMathOperator{\Ker}{Ker}

\DeclareMathOperator{\Hom}{Hom}
\DeclareMathOperator{\rk}{rk}
\DeclareMathOperator{\odd}{odd}
\DeclareMathOperator{\even}{even}

\DeclareMathOperator{\Cy}{Cy}
\DeclareMathOperator{\St}{St}

\DeclareMathOperator{\Sunn}{Sun}
\DeclareMathOperator{\Net}{Net}

\DeclareMathOperator{\QR}{QR}
\DeclareMathOperator{\diag}{diag}
\DeclareMathOperator{\Fl}{Fl}
\DeclareMathOperator{\diff}{diff}
\DeclareMathOperator{\Pe}{Pe}
\DeclareMathOperator{\girth}{girth}

\DeclareMathOperator{\Cl}{Cl}
\DeclareMathOperator{\Gr}{Gr}
\DeclareMathOperator{\pr}{pr}
\DeclareMathOperator{\Hilb}{Hilb}
\DeclareMathOperator{\Inter}{Inter}

\DeclareMathOperator{\inv}{inv}
\DeclareMathOperator{\adi}{seam}

\newcommand{\ca}[1]{\mathcal{#1}}
\newcommand{\Hr}{\widetilde{H}}
\newcommand{\dd}{\partial}

\newcommand{\I}{\mathbb{I}}

\def\Co{\mathbb C}
\def\Ro{\mathbb R}
\def\Qo{\mathbb Q}
\def\Zo{\mathbb Z}

\def\Zt{\mathbb Z_2}

\newcounter{stmcounter}[section]
\newcounter{thcounter}
\newcounter{problcounter}

\numberwithin{equation}{section}

\theoremstyle{plain}
\newtheorem{cor}[stmcounter]{Corollary}

\newtheorem{thm}[thcounter]{Theorem}

\newtheorem{prop}[stmcounter]{Proposition}
\newtheorem{lem}[stmcounter]{Lemma}

\newtheorem{que}[problcounter]{Question}

\theoremstyle{definition}
\newtheorem{defin}[stmcounter]{Definition}

\theoremstyle{remark}
\newtheorem{ex}[stmcounter]{Example}
\newtheorem{rem}[stmcounter]{Remark}
\newtheorem{con}[stmcounter]{Construction}

\newtheorem{algor}[stmcounter]{Algorithm}

\begin{document}
\title{Topological approach to diagonalization algorithms}

\author{Anton Ayzenberg}
\address{Faculty of computer science, National Research University Higher School of Economics, Russian Federation}
\email{ayzenberga@gmail.com}

\author{Konstantin Sorokin}
\address{Faculty of computer science, National Research University Higher School of Economics, Russian Federation}
\email{mopsless7@gmail.com}

\date{\today}
\thanks{The article was prepared within the framework of the HSE University Basic Research Program}



\begin{abstract}
In this paper we prove that there exists an asymptotical diagonalization algorithm for a class of sparse Hermitian (or real symmetric) matrices if and only if the matrices become Hessenberg matrices after some permutation of rows and columns. The proof is based on Morse theory, Roberts' theorem on indifference graphs, toric topology, and computer-based homological calculations.
\end{abstract}

\maketitle


\section{Introduction}\label{secIntro}

In this paper we address the following question which will soon be explained in detail.

\begin{que}
Which sparseness types of Hermitian (or real symmetric) matrices can be diagonalized by QR-type algorithms?
\end{que}

Let $M_\lambda$ denote the set of all Hermitian matrices of size $n$ with the given spectrum $\lambda=\{\lambda_1,\ldots,\lambda_n\}$. The classical QR-algorithm can be viewed as a cascade (dynamical system with discrete time) generated by $\QR\colon M_\lambda\to M_\lambda$ with the property that
\[
\lim_{n\to+\infty}\QR^n(A)=\diag(\lambda_{\sigma(1)},\ldots,\lambda_{\sigma(n)})
\]
for any initial matrix $A\in M_\lambda$ and some permutation $\sigma\in\Sigma_n$. There also exists a continuous version of the QR-algorithm, namely the flow of the full symmetric Toda lattice: $\dot{A}=[A,P(A)]$, where $P(L)$ is the antisymmetrization of $L$, see~\cite{Chu}. It is known that, for a simple spectrum $\lambda$ (which means there are no multiple eigenvalues $\lambda_i$), the space $M_\lambda$ is diffeomorphic to the variety $\Fl(\Co^n)$ of full flags in $\Co^n$. It is also known (see~\cite{BBR,BG,dMP,ChShSo}) that there exists a Riemannian metric $g$ and a Morse function $f$ on $M_\lambda$ such that the Toda flow is the gradient flow of~$f$. The stationary points of the Toda flow (the critical points of $f$) are again diagonal matrices with spectrum $\lambda$.

We consider the submanifolds of $M_\lambda$ given by some sparse forms of matrices. By a sparse matrix we mean a matrix with vanishing condition for some off-diagonal entries. It is convenient to encode the sparseness type by a simple graph $\Gamma$ with the vertex set $[n]=\{1,2,\ldots,n\}$ and an edge set $E_\Gamma$. A matrix $A=(a_{ij})$ is called \emph{$\Gamma$-shaped} if $a_{ij}=0$ for all $\{i,j\}\notin E_\Gamma$. Consider the space $M_{\Gamma,\lambda}$ of all $\Gamma$-shaped Hermitian matrices with spectrum $\lambda$. By Sard's lemma, the space $M_{\Gamma,\lambda}$ is a smooth manifold for generic $\lambda$. Although the precise description of all $\lambda$ for which $M_{\Gamma,\lambda}$ is smooth is unknown in general, in the following we will always assume that spectra are simple. 

There are particular cases of sparse matrices playing important roles in applications, namely Hessenberg (or staircase) matrices, where nonzero entries are allowed in a contiguous interval adjacent to main diagonal. It is convenient to encode staircase form by a Hessenberg function. A function $h\colon [n]\to[n]$ is called \emph{a Hessenberg function} if $h(i)\geqslant i$ for each $i\in[n]$, and $h(1)\leqslant h(2)\leqslant \cdots\leqslant h(n)$. Let $\Gamma(h)$ be the graph on the set $[n]$ with the edge set $E_h=\{(i,j)\mid i\leqslant j\leqslant h(i)\}$. Then $\Gamma(h)$-shaped matrices are the matrices $A$ whose entries $a_{ij}$ vanish for $j>h(i)$ (or $i>h(j)$). These matrices have staircase form determined by $h$: the value $h(i)$ encodes the lowest position at $i$-th column, where non-zero element is allowed.

It is well known that both the QR-algorithm and the flow of the full symmetric Toda lattice can be restricted to the submanifold $M_{\Gamma(h),\lambda}\subset M_\lambda$ of staircase matrices. This motivates the following definition.

\begin{defin}\label{definDiagonClass}
A class of $\Gamma$-shaped matrices is called \emph{a diagonalizable class} if there exists a Morse--Smale flow on a manifold $M_{\Gamma,\lambda}$ whose set of periodic trajectories is the discrete set of all diagonal matrices. In this case the graph $\Gamma$ is said to have \emph{diagonalizable type}.
\end{defin}

Instead of Morse--Smale flow in the definition one can use a Morse--Smale cascade (dynamical system with discrete time), see Remark~\ref{remSmaleCascades}.  

The existence of Toda flow on staircase matrices proves that $\Gamma(h)$ has diagonalizable type for each Hessenberg function $h$, see details in~\cite{AyzStaircase}. Notice that some matrices are not staircase, however they become staircase after relabeling rows and columns with the same permutation of indices. For example,
\begin{equation}\label{eqMatrixToAnother}
\begin{pmatrix}
  \ast & \ast & \ast \\
  \ast & \ast & 0 \\
  \ast & 0 & \ast
\end{pmatrix}
\stackrel{(2,3)}{\longrightarrow}
\begin{pmatrix}
  \ast & \ast & 0 \\
  \ast & \ast & \ast \\
  0 & \ast & \ast
\end{pmatrix},
\end{equation}
the matrix becomes staircase after permuting 2-nd and 3-rd rows and columns. In general, if $\sigma\in\Sigma_n$ is a permutation of the set $[n]$, and $\sigma\Gamma$ is the graph obtained from $\Gamma$ by relabelling the vertices with $\sigma$, there is a natural diffeomorphism
\begin{equation}\label{eqRelabelDiffeo}
\diff_\sigma\colon M_{\Gamma,\lambda}\to M_{\sigma\Gamma,\lambda}
\end{equation}
given by $\diff_\sigma(A)=P_\sigma AP_{\sigma}^{-1}$ for the permutation matrix $P_\sigma$. Since a manifold remains diffeomorphic under the permutation action, the diagonalizable property of a graph does not depend on a particular labelling of vertices: it depends only on the isomorphism class of a graph. 

Consequently, all graphs isomorphic to $\Gamma(h)$ for a Hessenberg function $h$, have diagonalizable type. The purpose of this paper is to prove the converse.

\begin{thm}\label{thmMainDtypeChar}
A class of $\Gamma$-shaped matrices is diagonalizable if and only if $\Gamma$ is isomorphic to a Hessenberg graph $\Gamma(h)$ for some Hessenberg function $h$.
\end{thm}

There is a name for graphs isomorphic to Hessenberg graphs.

\begin{defin}\label{definIndifGraph}
A graph $\Gamma$ is called \emph{an indifference graph} (or \emph{a unit interval graph}, or \emph{a proper interval graph}) if it is the intersection graph of some collection of closed unit intervals on a line $\Ro$.
\end{defin}

So far, $\Gamma=([n],E_\Gamma)$ is an indifference graph if and only if there exists a set of points $x_1,\ldots,x_n\in\Ro$ on a line such that $\{i,j\}\in E_\Gamma \Leftrightarrow |x_i-x_j|\leqslant 1$. We refer to the work of Roberts~\cite{RobPsych} who introduced the notion and the term.

\begin{prop}[\cite{Mertz}]\label{propMertzios}
A graph $\Gamma$ is isomorphic to $\Gamma(h)$ for some Hessenberg function $h$ if and only if $\Gamma$ is an indifference graph.
\end{prop}

If, in the definition of an indifference graph, we have $x_1<\cdots<x_n$, then $\Gamma=\Gamma(h)$ so the order of points correspond to the correct order of vertices inducing a Hessenberg function.

It follows from Proposition~\ref{propMertzios}, that all indifference graphs determine diagonalizable matrix classes as explained above. The nontrivial part of Theorem~\ref{thmMainDtypeChar} is therefore the following statement.

\begin{prop}\label{propMainOneSide}
If $\Gamma$ is not an indifference graph, then $\Gamma$ is not of diagonalizable type.
\end{prop}

The proof of Proposition~\ref{propMainOneSide} breaks into the following steps.
\begin{enumerate}
  \item Characterization of indifference graphs in terms of forbidden induced subgraphs obtained by Roberts~\cite{Roberts}. The forbidden subgraphs are cycle graphs $\Cy_k$ with $k\geqslant 4$, the claw graph $\St_3$, the 3-sun graph $\Sunn$, and the net graph $\Net$, see Fig.~\ref{figRoberts} below.
  \item Morse inequalities. In order to prove that $M_{\Gamma,\lambda}$ does not support a Morse--Smale flow with $n!$ critical points, it is sufficient to prove that the total Betti number $\rk H_*(M_{\Gamma,\lambda};R)$ is greater than $n!$ for some coefficient ring $R$.
  \item Each manifold $M_{\Gamma,\lambda}$ carries a natural compact torus action. We observe that the inequality $\rk H_*(M_{\Gamma,\lambda};R)>n!$ holds if and only if the manifold $M_{\Gamma,\lambda}$ is not \emph{cohomologically equivariantly formal} over $R$ in the sense of Goresky--Kottwitz--MacPherson. The results of toric topology allow to simplify the proof of non-equivariant formality. To prove the general result, it is sufficient to prove that forbidden subgraphs from item 1 produce non-equivariantly formal manifolds.
  \item Non-equivariant formality in case of $\Cy_k$, $k\geqslant 4$, and $\St_3$ was proved in~\cite{AyzStaircase} and~\cite{AyzArrows} respectively, and got a uniform explanation in terms of graphicahedra and cluster-permutohedra in the recent paper~\cite{AyzBuchGraph}. Graphicahedron~\cite{Graphicahedron} is a certain finite poset associated with a graph and cluster-permutohedron is its core in the sense of finite topology. Cluster-permutohedron $\Cl_\Gamma$ is the face poset of the torus action on $M_{\Gamma,\lambda}$. To prove non-formality of $M_{\Gamma,\lambda}$, we apply the result of~\cite{AyzMasSolo}, which states that face posets of equivariantly formal actions satisfy certain acyclicity conditions. The main observation of~\cite{AyzBuchGraph} is that graphicahedra (and cluster-permutohedra) of $\Cy_k$, $k\geqslant 4$, and $\St_3$ have nontrivial simplicial cohomology in degree $1$ and this contradicts equivariant formality of the corresponding matrix manifolds.
  \item In the current paper we finalize the proof of Proposition~\ref{propMainOneSide} by proving non-formality of isospectral matrix manifolds corresponding to $\Sunn$ and $\Net$. It happens that the idea used for $\Cy_k$ and $\St_3$ also works in these cases, however one has to compute 3-rd homology groups of the corresponding cluster-permutohedra. This cannot be done by hand, however the problem is solvable by a script in Sage~\cite{AyzCode}. 
  \item We also developed an alternative and potentially more general computational approach to prove non-formality which independently confirms our result. Conceptually, instead of looking at ordinary cohomology of cluster-permutohedra, we look at the cohomology of certain sheaves over these posets, namely the GKM-sheaves. Originally, such sheaves were defined over equivariant 1-skeleta of torus actions~\cite{Baird}, but there exists a natural way to extend them to higher dimensional structures: the face posets of torus actions. If a torus action is equivariantly formal, the GKM-sheaf satisfies certain homological properties. The most basic and general property is Atiyah--Bredon--Franz--Puppe (ABFP) exact sequence~\cite{FP}: this is the strengthening of Chang--Skjelbred theorem~\cite{ChSk} (which is, in turn, the principal tool used in GKM-theory~\cite{GKM}). Checking exactness of ABFP sequence is an algorithmic task. Although we could not treat the whole ABFP sequence of $M_{\Sunn,\lambda}$ and $M_{\Net,\lambda}$ due to extremely high computational complexity of this problem, we were able to make calculations which contradict to the exactness of the ABFP sequence. This approach proves non-formality of all manifolds $M_{\St_3,\lambda}$, $M_{\Cy_k,\lambda}$, $M_{\Net,\lambda}$, $M_{\Sunn,\lambda}$ as well.
\end{enumerate}

It should be noticed that the study of topology of isospectral matrix manifolds of various types using integrable dynamical systems is a common area of research, see~\cite{Tomei,dMP,Penskoi}. In this paper we solve a somewhat opposite task: we apply topology to prove that dynamical systems with certain properties do not exist.

The paper has the following structure. In Section~\ref{secDefResults} we give all the required definitions and details missing in the introduction. In Section~\ref{secKnownToricCySt} we review two proofs of non-formality of matrix manifolds corresponding to $\Cy_k, k\geqslant 4$ and $\St_3$. In Section~\ref{secSunAndNet}, the basics of ABFP sequence and GKM theory are recalled. They are used to prove non-formality of manifolds corresponding to $\Net$ and $\Sunn$ with computer-aided experiments. In Section~\ref{secLast}, we observe that all arguments of the paper remain valid for the real versions of $M_{\Gamma,\lambda}$: the manifolds of isospectral real symmetric matrices. In the real case, instead of torus actions we have discrete 2-torus actions, and several recent results of real toric topology can be applied. Finally, in Section~\ref{secLast} we introduce a new graph invariant motivated by the current study, which can be applied in the design of diagonalization algorithms.

\section{Definitions, results, and basic steps of proof}\label{secDefResults}

Let $M_n$ denote the vector space of all Hermitian matrices of size $n$. We have $\dim_\Ro M_n=n^2$. For a given set
$\lambda = \{\lambda_1,\ldots,\lambda_{n}\}$ of pairwise distinct real numbers consider the subset $M_{\lambda}\subset M_n$ of all
matrices with eigenvalues $\{\lambda_1,\ldots,\lambda_{n}\}$ where it is assumed that $\lambda_1<\lambda_2<\cdots<\lambda_{n}$. Let $U(n)$ be the group of unitary matrices and $T^{n}\subseteq U(n)$ be the compact torus of diagonal unitary matrices
\[
T^{n}=\left\{\diag(t_1,\ldots,t_{n}), t_i\in \Co, |t_i|=1\right\}
\]
The group $U(n)$ acts on $M_{n}$ by conjugation: this is essentially the coadjoint representation of $U(n)$. It easily follows that $M_\lambda$ is diffeomorphic to the manifold $\Fl(\Co^n)=U(n)/T^{n}$ of full complex flags, since both are homogeneous spaces of $U(n)$ with the same stabilizer $T^n$. Moreover, there is a trivial smooth fibration
\begin{equation}\label{eqGmap}
g\colon M_n\setminus\Sigma\to C,
\end{equation}
where $\Sigma$ is the set of Hermitian matrices with multiple eigenvalues, $C=\{(\lambda_1,\ldots,\lambda_n)\in\Ro^{n}\mid
\lambda_1<\cdots<\lambda_n\}$ is an open Weyl chamber, and $g$ maps the matrix to its eigenvalues, listed in the increasing order. The fiber of $g$ over a point $\lambda$ is the manifold $M_\lambda$. We have $\dim_\Ro M_{\lambda}=n(n-1)$. The group $T^{n}$ acts on $M_{n}$ by conjugation: $A\mapsto DAD^{-1}$. In coordinate notation we have
\begin{equation}\label{eqActionExplicit}
(a_{ij})_{\substack{i=1,\ldots,n\\j=1,\ldots,n}}\mapsto
(t_it_j^{-1}a_{ij})_{\substack{i=1,\ldots,n\\j=1,\ldots,n}}
\end{equation}
Scalar matrices commute with every matrix $A$, therefore the diagonal subgroup of the torus acts non-effectively. The fixed points
of the $T^n$-action on $M_{\lambda}$ are diagonal matrices with spectrum $\lambda$. These are the diagonal matrices of the form
$A_\sigma=\diag(\lambda_{\sigma(1)},\lambda_{\sigma(2)},\ldots,\lambda_{\sigma(n)})$ for all possible permutations $\sigma\in \Sigma_{n}$.

Let $\Gamma$ be a simple graph by which we mean a finite graph without multiple edges and loops, on the vertex set $[n]$, and an edge set $E_\Gamma$. A graph $\Gamma$ is assumed connected unless stated otherwise. 

\begin{defin}\label{definGammaShaped}
Consider the vector subspace of Hermitian matrices
\begin{equation}\label{eqMGamma}
M_\Gamma=\{A\in M_{n}\mid a_{ij}=0, \mbox{ if } \{i,j\}\notin E_\Gamma\}.
\end{equation}
Matrices from $M_\Gamma$ are called \emph{$\Gamma$-shaped}. Consider \emph{the subspace of isospectral $\Gamma$-shaped matrices}:
\begin{equation}\label{eqMGammaLambda}
M_{\Gamma,\lambda}=M_{\Gamma}\cap M_{\lambda}\subset M_n.
\end{equation}
\end{defin}

\begin{rem}\label{remSard}
According to Sard's theorem, the subset $M_{\Gamma,\lambda}$ is a smooth submanifold of the vector space $M_\Gamma\cong \Ro^{n+2|E_\Gamma|}$ for generic $\lambda$. Indeed, all regular values of the map $g|_{M_\Gamma}\colon M_\Gamma\to C$, see~\eqref{eqGmap}, determine smooth isospectral submanifolds $M_{\Gamma,\lambda}=(g|_{M_\Gamma})^{-1}(\lambda)$. Let $U$ be the set of regular values, it is open and dense. In the following, we always assume that $\lambda$ is simple and generic, i.e. lies in $U$, so that $M_{\Gamma,\lambda}$ is a smooth compact manifold. The precise description of $U$ depends on a graph $\Gamma$, and in some cases $U$ may be nontrivial, as shown in~\cite{AyzPeriodic} for periodic tridiagonal matrices. However, the precise description of the set of regular values $U$ is irrelevant to our current study.
\end{rem}

\begin{rem}\label{remOrientableParallel}
For generic $\lambda$, the manifold $M_{\Gamma,\lambda}$ is given by a nondegenerate system of $n$ equations in $M_\Gamma\cong\Ro^{n+2|E_\Gamma|}$. Therefore this manifold is normally parallelizable. In particular, $M_{\Gamma,\lambda}$ is orientable and its tangent characteristic classes vanish.
\end{rem}

The $T^n$-action on $M_n$ preserves all subsets $M_\Gamma$, $M_\lambda$, $M_{\Gamma,\lambda}$. In particular, we have a canonical smooth $T^n$-action on $M_{\Gamma,\lambda}$. Simple count of parameters implies
\begin{equation}\label{eqDimM}
\dim_{\Ro} M_{\Gamma,\lambda}=2|E_\Gamma|.
\end{equation}
%

We now review some basic facts from the theory of compact torus actions on manifolds. Let $R$ denote a coefficient ring (either $\Zo$ or a field). Assume that a torus $T=T^k$ acts on a smooth closed manifold $X$. Let $ET\to BT$ be the classifying principal $T$-bundle, and $X_T=X\times_TET$ be the Borel construction of $X$. We have a Serre fibration $p\colon X_T\stackrel{X}{\to}BT$. The cohomology ring $H^*_T(X;R)=H^*(X_T;R)$ is called \emph{the equivariant cohomology ring} of~$X$. Via the induced map $p^*\colon H^*(BT;R)\to H^*(X_T;R)$, the equivariant cohomology attain the natural structure of a graded module over $H^*(BT;R)\cong R[k]$, the polynomial ring in $k=\dim T$ generators of degree $2$. Since $BT$ is simply connected, the fibration $p$ induces Serre spectral sequence:
\begin{equation}\label{eqSerreSpSec}
E_2^{p,q}\cong H^p(BT;R)\otimes H^q(X;R)\Rightarrow H_T^{p+q}(X;R).
\end{equation}
The $T$-action on $X$ is called \emph{cohomologically equivariantly formal} (over $R$) if~\eqref{eqSerreSpSec} collapses at $E_2$. The following characterization of equivariant formality is known.

\begin{lem}\label{lemEquivFormFixedPoints}
Consider a smooth $T$-action on $X$, such that $X^T$ is finite and nonempty, and let $R$ be either $\Zo$ or a field. Then the following conditions are equivalent
\begin{enumerate}
  \item The $T$-action on $X$ is cohomologically equivariantly formal over $R$.
  \item $H^{\odd}(X;R)=0$.
  \item $H^*_T(X;R)$ is a free $H^*(BT;R)$-module.
\end{enumerate}
\end{lem}

This lemma was proved in \cite[Lm.2.1]{MasPan} for $R=\Zo$ and the proof for fields follows the same lines. It is known that the Euler characteristic $\chi(X)$ equals the number $\#X^T$ of fixed points for torus actions with isolated fixed points (see e.g.~\cite[Ch.III]{Bred}). Therefore, if $R$ is a field (so the torsion can be neglected), the following conditions are equivalent:
\begin{enumerate}
  \item the $T$-action on $X$ is equivariantly formal;
  \item the total Betti number equals the number of fixed points
  \begin{equation}\label{eqTotalBettiNumberFormal}
  \dim H_*(X;R) = \chi(X) = \#X^T.
  \end{equation}
\end{enumerate}
If the action is not equivariantly formal, we have
\begin{equation}\label{eqIneqBetti}
\dim H_*(X) = \dim H_{\even}(X)+\dim H_{\odd}(X)=\chi(X)+2\dim H_{\odd}(X)>\#X^T.
\end{equation}

In~\cite{AyzStaircase} we proved that $M_{\Gamma(h),\lambda}$ admits Morse function with all critical points having even index, which implies that there is a Morse decomposition with even-dimensional cells. Therefore $H^{\odd}(M_{\Gamma(h),\lambda};R)$ for any ring $R$, so these manifolds are equivariantly formal. Therefore, diffeomorphism~\eqref{eqRelabelDiffeo} and Proposition~\ref{propIndiffCharact} imply

\begin{prop}
If $\Gamma$ is an indifference graph, then the canonical $T$-action on $M_{\Gamma,\lambda}$ is equivariantly formal.
\end{prop}

In this paper we prove the converse.

\begin{thm}\label{thmNotEquivFormal}
If $\Gamma$ is not an indifference graph, then $M_{\Gamma,\lambda}$ is not equivariantly formal over $\Zo$, $\Qo$, and $\Zt$.
\end{thm}

The proof is based on the following sequence of statements. The first statement is a well-known result in the intersection graph theory, proved by Roberts~\cite{Roberts} (see also~\cite[Exer.3.12]{McMc}).

\begin{prop}[\cite{Roberts}]\label{propIndiffCharact}
A graph $\Gamma$ is an indifference graph if and only if $\Gamma$ does not contain induced subgraphs of the types shown on Fig.~\ref{figRoberts}: (1) the cycle graphs $\Cy_k$ with $k\geqslant 4$ vertices; (2) the claw graph, also known as $3$-star graph $\St_3$; (3) the net graph $\Net$; (4) the 3-sun graph $\Sunn$.
\end{prop}

\begin{figure}[h]
\begin{center}
\includegraphics[scale=0.35]{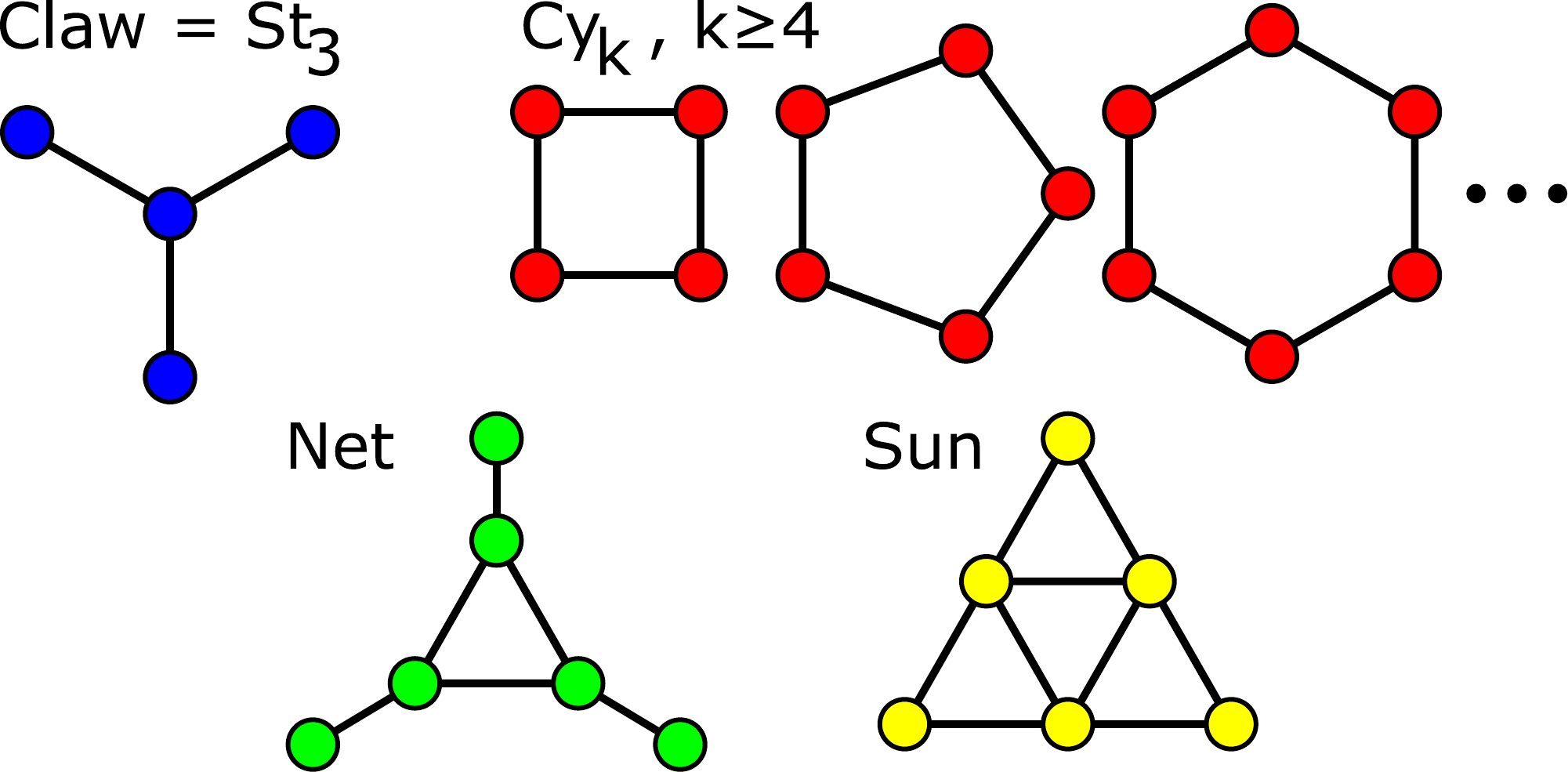}
\end{center}
\caption{Forbidden subgraphs for the class of indifference graphs}\label{figRoberts}
\end{figure}

We recall that the induced subgraph of $\Gamma$ on a vertex subset $B\subset[n]$ is the subgraph $\Gamma_B$ which contains all edges of $\Gamma$ incident to vertices from $B$. For example, the complete graph $K_4$ contains the claw $\St_3$ as a subgraph, but not as an induced subgraph. It follows that whenever $\Gamma$ is not an indifference graph, it contains one of the graphs: $\Cy_k$ ($k\geqslant 4$), $\St_3$, $\Net$, $\Sunn$ as an induced subgraph. The next 3 lemmata are proved in Section~\ref{secKnownToricCySt}. 

\begin{lem}\label{lemInducedGeneral}
Assume that $\lambda$ is generic and $M_{\Gamma,\lambda}$ is equivariantly formal. Let $\Gamma_B$ be the induced subgraph of $\Gamma$ on a vertex subset $B\subset[n]$ and $\lambda_B$ be any subset of $\{\lambda_1,\ldots,\lambda_n\}$ of cardinality $|B|$. Then $M_{\Gamma_B,\lambda_B}$ is also equivariantly formal.
\end{lem}

This lemma shows that to prove Theorem~\ref{thmNotEquivFormal}, it is sufficient to prove non-formality of the matrix manifolds corresponding to the forbidden subgraphs in Fig.~\ref{figRoberts}. 

\begin{lem}\label{lemCycleGraphs}
For generic $\lambda$ and $k\geqslant 4$, the space $M_{\Cy_k,\lambda}$ is not equivariantly formal over any $R$.
\end{lem}

\begin{lem}\label{lemStarGraph}
For generic $\lambda$, the space $M_{\St_3,\lambda}$ is not equivariantly formal over any $R$.
\end{lem}

The proof of non-equivariant formality of $M_{\Gamma,\lambda}$ for the graphs $\Net$ and $\Sunn$ can be done similarly to $\Cy_k$ and $\St_3$. However, in these cases, the proof is much harder computationally, so we outline another approach to check ourselves. 
The next two lemmata are proved in Section~\ref{secSunAndNet}.

\begin{lem}\label{lemNetGraph}
For generic $\lambda$, the space $M_{\Net,\lambda}$ is not equivariantly formal over $\Zo$, $\Qo$, and $\Zt$.
\end{lem}

\begin{lem}\label{lemSunGraph}
For generic $\lambda$, the space $M_{\Sunn,\lambda}$ is not equivariantly formal over $\Zo$, $\Qo$, and $\Zt$.
\end{lem}

These lemmata prove Theorem~\ref{thmNotEquivFormal}. Proposition~\ref{propMainOneSide} together with the main Theorem~\ref{thmMainDtypeChar} follow from Theorem~\ref{thmNotEquivFormal} using Morse arguments as follows.

\begin{proof}[Proof of Proposition~\ref{propMainOneSide} and Theorem~\ref{thmMainDtypeChar}]
If there exists a Morse--Smale system on a compact closed manifold where every trajectory has a limiting stationary point of a hyperbolic type, then there exists a gradient Morse flow with the same indices of stationary points as proved by Smale~\cite{Smale}. Morse inequalities imply that any Morse flow on $M_{\Gamma,\lambda}$ has at least $\dim H_*(M_{\Gamma,\lambda})$ stationary points. Since $M_{\Gamma,\lambda}$ is not equivariantly formal (at least for some field $R$), we have
\[
\dim H_*(M_{\Gamma,\lambda};R)>\#M_{\Gamma,\lambda}^T=n!,
\]
according to~\eqref{eqIneqBetti}. Hence $M_{\Gamma,\lambda}$ cannot have $n!$ stationary points, thus violating Definition~\ref{definDiagonClass}.
\end{proof}

\begin{rem}\label{remSmaleCascades}
Instead of Morse--Smale flows one can use Morse--Smale cascades (dynamical systems with discrete time) both in Definition~\ref{definDiagonClass} and in the proof above. Morse inequalities hold for such systems, see~\cite{Smale1} and~\cite{Smale2}. The discrete time setting is more natural if one speaks about QR-type algorithms instead of Toda flows. Notice that the classical QR-algorithm is a Morse--Smale cascade on the manifold $M_\lambda$ of isospectral matrices. Indeed, QR-algorithm can be treated as a modified version of the Toda flow sampled at integer times, see~\cite{Chu}, while the latter Toda flow is a Morse--Smale system on $M_\lambda$, as follows from the study of its center manifolds in the same paper.
\end{rem}

\section{The Cycles, the Claw, and the topology of graphicahedra}\label{secKnownToricCySt}

\subsection{Review of known results}

Let us recall the following result from~\cite{MasPan}.

\begin{lem}[{\cite[Lem.2.2]{MasPan}}]\label{lemInvarIsFormal}
Let $T$ act on $X$, and $Y$ be a connected component of the fixed point set $X^H$ for some closed subgroup $H\subseteq T$. Then condition $H^{\odd}(X)=0$ implies $H^{\odd}(Y)=0$ and $Y^T\neq\varnothing$. Equivariant formality of $X$ implies equivariant formality of $Y$.
\end{lem}

We use it in a natural way to prove Lemma~\ref{lemInducedGeneral}.

\begin{proof}[Proof of Lemma~\ref{lemInducedGeneral}]
Consider the coordinate subtorus $H=T^{[n]\setminus B}$ of the torus $T^n$ acting on $M_{\Gamma,\lambda}$. According to the expression~\eqref{eqActionExplicit}, the matrix $A\in M_{\Gamma,\lambda}$ is fixed by $H$ if and only if its off-diagonal elements $a_{ij}$ vanish whenever either $i$ or $j$ belongs to $[n]\setminus B$. Therefore $A$ has a block form, with a big block $A_B$ corresponding to the index set $B$, and all other blocks of unit size.

The fixed point submanifold $M_{\Gamma,\lambda}^H$ have connected components defined by collections $\lambda_B$ of eigenvalues, which live in the block $A_B$. Each connected component of $M_{\Gamma,\lambda}^H$ is therefore diffeomorphic to $M_{\Gamma_B,\lambda_B}$ for some subset $\lambda_B\subset\lambda$. Lemma~\ref{lemInvarIsFormal} finishes the proof.
\end{proof}


\begin{proof}[Proof of Lemma~\ref{lemCycleGraphs}]
The topology of $M_{\Cy_k,\lambda}$ was described in detail in~\cite{AyzPeriodic}. Although the paper contains the combinatorial formulae for Betti numbers of $M_{\Cy_k,\lambda}$, it is quite complicated to extract their precise values in general. We may use another approach. In~\cite{AyzPeriodic}, it was proved that whenever $M_{\Cy_k,\lambda}$ is a smooth manifold, there holds $\pi_1(M_{\Cy_k,\lambda})\cong\Zo^{k-3}$. This implies $H_1(M_{\Cy_k,\lambda};R)\neq 0$ for $k\geqslant 4$ and any coefficient ring $R$. According to Lemma~\ref{lemEquivFormFixedPoints}, this fact proves the lemma.
\end{proof}

\begin{proof}[Proof of Lemma~\ref{lemStarGraph}]
In~\cite{AyzArrows} we studied more general class of matrix spaces given by star graphs $\St_k$ with arbitrary number $k$ of rays. In particular, it was proved that $M_{\St_k,\lambda}$ is not equivariantly formal for $k\geqslant 3$. For the particular case $k=3$ we computed all Betti numbers:
\[
(\beta_0,\beta_1,\ldots,\beta_6)=(1,1,12,0,12,1,1)
\]
independently of the coefficient ring $R$. These results are based on the study of the orbit space $M_{\St_3,\lambda}/T^3$ which is proved to be homeomorphic to $D^2\times S^1$.
\end{proof}

\begin{rem}
Notice that $M_{\St_3,\lambda}$ is a torus manifold, which means that the dimension of the acting torus equals half the real dimension of the manifold. Such actions are well studied. The orbit space criterion of equivariant formality of torus manifolds was proved in~\cite{MasPan}. This criterion implies in particular, that the orbit space of any equivariantly formal torus manifold is a homology disk (see Remark~\ref{remOnMasPanovHomologyCells} below). Since $M_{\St_3,\lambda}/T^3$ is not a disk, this fact already implies non-formality.
\end{rem}

\subsection{Face posets of torus actions}

Let us formulate several other approaches to study general torus actions, as well as particular torus actions on manifolds $M_{\Gamma,\lambda}$.

\begin{con}\label{conInvarFacesEtc}
Consider a smooth action of a compact torus $T$ on a closed manifold $X$, having isolated fixed points. The details of the following construction, the missing proofs and references can be found in~\cite{AyzCherep} and~\cite{AyzMasSolo}.

For any connected closed subgroup $H\subseteq T$ we consider the subset $X^H$ fixed by $H$, this is a closed smooth submanifold of $X$. Connected components of $X^H$ are called invariant submanifolds. An invariant submanifold $Y$ is called \emph{a face submanifold}, if it contains a $T$-fixed point (that is $Y\cap X^T\neq\varnothing$). Each face submanifold is $T$-stable. Its orbit space by the $T$-action is called \emph{a face}. We denote a face by a letter $F$, while the corresponding face submanifold is denoted $X_F$, so that there holds $X_F/T=F$. Let $T_F\subseteq T$ denote the noneffective kernel of the $T$-action on $X_F$ (this can be treated as the stabilizer of a generic point of $X_F$). So far, we have the effective action of $T/T_F$ on $X_F$. The dimension $\dim T/T_F$ is called the rank of $F$ (or the rank of $X_F$) and denoted by $\rk F$.

All face submanifolds (or all faces) are ordered by inclusion. They form a finite poset graded with the rank function. We denote this poset by $S(X)$. The poset has the greatest element $\hat{1}$, the manifold $X$ itself. The minimal elements are the fixed points of the action, they have rank 0.
\end{con}

\begin{con}\label{conWeights}
If $x\in X^T$ is an isolated fixed point of a torus action, the tangent representation $T_xX$ decomposes into a sum of irreducible representations. All irreducible representations of a torus have real dimension 2 (unless they are trivial), so that we have
\[
T_xX\cong V(\alpha_{x,1})\oplus\cdots\oplus V(\alpha_{x,n}),
\]
where $\alpha_{x,i}\in\Hom(T;T^1)\cong\Zo^{\dim T}$ are determined up to sign and called \emph{the tangent weights} at $x$. Here $V(\alpha)\cong\Co$ is the irreducible representation given by $t\cdot z=\alpha(t)z$ for $\alpha\in\Hom(T;T^1)$.

We say that \emph{an action is $j$-independent}, if, for any isolated point $x\in X^T$, any $\leqslant j$ of its tangent weights are linearly independent over $\Qo$ (this means linear independency in the vector space $\Hom(T;T^1)\otimes\Qo\cong \Qo^{\dim T}$).
\end{con}

In~\cite{AyzCherep}, we studied the properties of posets $S(X)$ arising from general torus actions. In particular, it was proved that, for any element $s\in S(X)$, the upper ideal $S(X)_{\geqslant s}$ is a geometric lattice. In case $x\in X^T$ is a fixed point, the ideal $S(X)_{\geqslant x}$ is isomorphic to the lattice of flats of the linear matroid corresponding to the collection of tangent weights at~$x$.

For equivariantly formal torus actions, the poset of faces exhibits nice acyclicity properties, similar to Cohen--Macaulayness of simplicial complexes corresponding to toric varieties. The poset $S(X)$ itself is not interesting from topological point of view, since it has the greatest element, so its geometrical realization $|S(X)|$ is a cone, hence contractible. However, we can restrict to its ``skeleta'' $S(X)_r=\{s\in S(X)\mid \rk s\leqslant r\}$. In~\cite{AyzMasSolo} we proved the following

\begin{prop}[{\cite[Thm.1]{AyzMasSolo}}]\label{propAcyclicityOfPoset}
Consider a $T$-action on a manifold $X$. Assume that one of the two conditions is satisfied:
\begin{itemize}
  \item all stabilizers of the action are connected;
  \item $R=\Qo$.
\end{itemize}
Also assume that the action is equivariantly formal over $R$ and $j$-independent. Then, for any $r$, the poset $S(X)_r$ is $\min(j+1,r-1)$-acyclic, that is
\[
\Hr_i(|S(X)_r|;R)=0\mbox{ for } i\leqslant \min(j+1,r-1).
\]
\end{prop}

\begin{ex}\label{exAcyclGKMactions}
An important class of torus actions is given by GKM-actions, see Definition~\ref{definGKMmfd}. They are always $2$-independent and equivariantly formal, so that we get $3$-acyclicity of their skeleta $S(X)_r$ (unless the dimension $r=\dim |S(X)_r|$ is $3$ or lower).
\end{ex}

\subsection{Cluster-permutohedra}
The general review of cluster-permutohedra as they appear in the study of isospectral matrix spaces, as well as their relation to the (very similar) notion of graphicahedra, is given in~\cite{AyzBuchGraph}. Here we give the necessary definitions and recall the results needed for the calculations to follow.

\begin{con}\label{conClusterPerm}
As before, let $\Gamma$ be a connected graph on the vertex set $V_\Gamma$, $|V_\Gamma|=n$ with an edge set $E_\Gamma$. We call an unordered subdivision $\ca{C}=\{V_1,\ldots,V_k\}$ of the set $V_\Gamma$ a \emph{clustering}, if each induced subgraph $\Gamma_{V_i}$ is connected (these components are ``the clusters'' which explains the name). The set $\ca{L}_\Gamma$ of all clusterings is partially ordered by refinement: $\ca{C}'\leq \ca{C}$ if each $V_i'\in \ca{C}'$ is a subset of some $V_j\in \ca{C}$. The poset $\ca{L}_\Gamma$ has the least element $\hat{0}=\{\{1\},\ldots,\{n\}\}$ and the greatest element $\{V_\Gamma\}$, it is graded by $\rk\ca{C}=n-1-|\ca{C}|$. It can be seen that the poset $\ca{L}_\Gamma$ is a geometric lattice. Indeed, $\ca{L}_\Gamma$ is the lattice of flats of the graphical matroid corresponding to the graph $\Gamma$.

Now we consider all possible bijections $p$ from $V_\Gamma$ to $[n]=\{1,\ldots,n\}$. We say that two bijections $p_1,p_2\colon V\to[n]$ are equivalent with respect to a clustering $\ca{C}$ (or simply $\ca{C}$-\emph{equivalent}), denoted $p_1\stackrel{\ca{C}}{\sim}p_2$, if $p_1,p_2$ differ by permutations within clusters $V_i$ of $\ca{C}$. In other words,
\[
p_1^{-1}\circ p_2 \in \Sigma_{\ca{C}}=\Sigma_{V_1}\times\cdots\times\Sigma_{V_k}\subseteq\Sigma_{V_\Gamma}.
\]
The class of $\ca{C}$-equivalent bijections will be called \emph{an assignment} for the clustering $\ca{C}$. Assignments for $\ca{C}$ are naturally identified with the cosets $\Sigma_V/\Sigma_{\ca{C}}$. Let $\Cl_\Gamma$ denote the set of all possible pairs $(\ca{C},A)$ where $\ca{C}\in \ca{L}_\Gamma$ is a clustering, and $A\in \Sigma_V/\Sigma_{\ca{C}}$ is an assignment for this clustering. Notice that any refinement $\ca{C}'\leq \ca{C}$ induces the inclusion of subgroups $\Sigma_{\ca{C}'}\hookrightarrow \Sigma_{\ca{C}}$ hence the natural surjection on the cosets
\[
\pr_{\ca{C}'\leq \ca{C}}\colon \Sigma_V/\Sigma_{\ca{C}'}\to \Sigma_V/\Sigma_{\ca{C}}.
\]
Define the partial order on the set $\Cl_\Gamma$ by setting $(\ca{C}',A')\leq (\ca{C},A)$ if and only if $\ca{C}'\leq \ca{C}$ and $\pr_{\ca{C}'\leq \ca{C}}(A')=A$. This poset is naturally graded by $\rk((\ca{C},A))=\rk\ca{C}$.
\end{con}

\begin{defin}
The poset $\Cl_\Gamma$ is called \emph{the cluster-permutohedron} of a graph $\Gamma$.
\end{defin}

\begin{ex}
If $\Gamma=\I_n$ is a simple path on $n$ vertices, the assignments for clusterings bijectively correspond to linearly ordered partitions of $[n]$. Hence $\Cl_{\I_n}$ is isomorphic to the poset of linearly ordered partitions of $[n]$. This poset is in turn isomorphic to the face poset of the classical \emph{permutohedron} $\Pe^{n-1}$, see~\cite[Ex.0.10]{Zieg}. This example explains the general name of cluster-permutohedra. 

A generalization of a permutohedron given by the poset of all cyclically ordered partitions of $[n]$ was introduced in the work of Panina~\cite{Panina} by the name \emph{cyclopermutohedron}. In our terms, this poset is $\Cl_{\Cy_n}$.
\end{ex}

\begin{rem}
For any $\Gamma$, there exists $n!$ minimal elements (that are elements of rank $0$) in $\Cl_{\Gamma}$, they correspond to cosets of the trivial subgroup: $\Sigma_{V_\Gamma}/1\cong \Sigma$. For any $\sigma$ of rank $0$, the upper order ideal $(\Cl_{\Gamma})_{\geqslant \sigma}$ is isomorphic to the geometric lattice $\ca{L}_\Gamma$ of a graphical matroid corresponding to $\Gamma$.
\end{rem}

There exists another construction: \emph{the graphicahedron} of a graph.

\begin{con}
Let $\Gamma=(V_\Gamma,E_\Gamma)$ be a graph as before. The graphicahedron $\Gr_{\Gamma}$, as a set, consists of pairs $(D,A)$, where $D\subseteq E_\Gamma$ is any set of edges, and $A$ is an assignment for the clustering, given by connected components of the subgraph $(V_\Gamma,D)$ of $\Gamma$. The order is induced from the natural inclusion order on $2^{V_\Gamma}$ similarly to Construction~\ref{conClusterPerm}.
\end{con}

Graphicahedra were introduced in~\cite{Graphicahedron} and studied further in~\cite{SymGraph}. In~\cite{AyzBuchGraph} we described the precise relations between graphicahedra and cluster-permutohedra, and understood that cluster-permutohedra are better suited for the tasks of toric topology.

\begin{rem}
In the case of graphicahedron, we still have $n!$ minimal elements corresponding to permutations of $V_\Gamma$. But in this case, for any minimal element $\sigma$, the upper order ideal $(\Gr_{\Gamma})_{\geqslant\sigma}$ is isomorphic to the boolean lattice $2^{E_\Gamma}$.
\end{rem}

\begin{rem}\label{remTreeClusterGraphic}
If $\Gamma$ is a tree, then the cluster-permutohedron $\Cl_\Gamma$ is isomorphic to the graphicahedron $\Gr_\Gamma$. On the other hand, if $\Gamma$ has cycles, the posets $\Cl_\Gamma$ and $\Gr_\Gamma$ are non-isomorphic, they even have different cardinalities.

However there exists a natural Galois insertion $\iota\colon \Cl_\Gamma\hookrightarrow\Gr_\Gamma$. In particular, for properly defined skeleta $(\Gr_\Gamma)_r$ and $(\Cl_\Gamma)_r$ this Galois insertion induces homotopy equivalences of the geometrical realizations.

On the level of $1$-skeleta, the maps $\iota\colon (\Cl_\Gamma)_1 \rightleftarrows (\Gr_\Gamma)_1\colon \rho$ are inverses of one another. Both 1-skeleta $(\Cl_\Gamma)_1\cong (\Gr_\Gamma)_1$ are isomorphic to the Cayley graph of the permutation group $\Sigma_V$ with the set of transpositions $\{(i,j)\mid \{i,j\}\in E_\Gamma\}$ taken as the generators' set. See~\cite[Thm.1]{AyzBuchGraph} for details.
\end{rem}

Our original motivation to introduce cluster-permutohedra in~\cite{AyzArrows} was the following statement, which provides the link to isospectral matrix spaces. It was proved in its full generality in~\cite{AyzBuchGraph}.

\begin{prop}[{\cite[Thm.1]{AyzBuchGraph}}]\label{propGirthAndGraphicahedron}
Let $\Gamma$ be a graph on $n$ vertices. Assume that the isospectral space $M_{\Gamma,\lambda}$ is smooth. Then the following holds for the torus action on this manifold.
\begin{enumerate}
  \item The face poset $S(M_{\Gamma,\lambda})$ is isomorphic to the cluster-permutohedron $\Cl_\Gamma$.
  \item Let $g$ denote the girth of $\Gamma$. Then the torus action on $M_{\Gamma,\lambda}$ is $(g-1)$-independent.
  \item All stabilizers of the torus action on $M_{\Gamma,\lambda}$ are connected.
\end{enumerate}
\end{prop}

Recall that the girth is the minimal length of cycles in $\Gamma$ (assumed $+\infty$, if $\Gamma$ is acyclic).

\subsection{Obstructions to equivariant formality in the topology of graphicahedra}

We already proved that $M_{\St_3,\lambda}$ and $M_{\Cy_k,\lambda}$, $k\geqslant 4$, are not equivariantly formal. However, there is another way to observe these facts coming from the known results on graphicahedra.

As mentioned in~\cite{SymGraph}, the graphicahedron $\Gr_{\St_3}$ corresponding to the claw graph $\St_3$ is isomorphic to the toroidal regular map $\{6,3\}_{(2,2)}$ in~\cite[Sect.8.4]{CoxMoser}. Therefore, the 2-skeleton $(\Gr_{\St_3})_2$ is homeomorphic to the 2-torus $T^2$, and hence $H_1((\Gr_{\St_3})_2)\neq 0$. Since $\St_3$ is a tree, Remark~\ref{remTreeClusterGraphic} implies that $(\Gr_{\St_3})_2\cong (\Cl_{\St_3})_2$, so graphicahedron is a torus as well. The poset $\Cl_{\St_3}$ is isomorphic to the poset of faces of the isospectral manifold $M_{\St_3,\lambda}$. Since $H_1((\Cl_{\St_3})_2)\neq 0$, Proposition~\ref{propAcyclicityOfPoset} implies that $M_{\St_3,\lambda}$ is not equivariantly formal.

The argument with the cycle graphs $\Cy_k$, $k\geqslant 4$ is pretty much similar. As shown in~\cite[Thm.8]{SymGraph} (and independently in~\cite{AyzPeriodic}), the poset $\Gr_{\Cy_k}$ is the face poset of a regular cell subdivision of the $(k-1)$-dimensional torus $T^{k-1}$. In particular, it follows that
\[
H_1((\Gr_{\Cy_k})_2)\neq 0\mbox{ for }k\geqslant 4.
\]
The homotopy equivalence between skeleta of graphicahedra and cluster-permutohedra observed in Remark~\ref{remTreeClusterGraphic} implies that $H_1((\Cl_{\Cy_k})_2)\neq 0$ as well, since $|(\Gr_{\Cy_k})_2|\simeq|(\Cl_{\Cy_k})_2|$. Since the 2-skeleton of the face poset of $M_{\Cy_k,\lambda}$ has nontrivial homology in degree $1$, Proposition~\ref{propAcyclicityOfPoset} again implies that $M_{\Cy_k,\lambda}$ is not equivariantly formal for $k\geqslant 4$.

\section{The Net, the Sun, and computer algebra}\label{secSunAndNet}

\subsection{Geometrical realizations of graphicahedra}

The arguments of the previous paragraph suggest the following strategy to prove that $M_{\Net,\lambda}$ and $M_{\Sunn,\lambda}$ are not equivariantly formal.
\begin{itemize}
  \item Construct cluster-permutohedra $\Cl_{\Net}$ and $\Cl_{\Sunn}$, and their rank-selected skeleta.
  \item Compute simplicial homology in low degrees of the skeleta.
  \item If at least one of the homology groups is nontrivial, then Proposition~\ref{propAcyclicityOfPoset} implies that the corresponding isospectral manifold is not equivariantly formal.
\end{itemize}

Notice that both graphs $\Net$ and $\Sunn$ have girth $3$, therefore the torus actions on the corresponding manifolds are $2$-independent by Proposition~\ref{propGirthAndGraphicahedron}. Proposition~\ref{propAcyclicityOfPoset} therefore assures $3$-acyclicity of the $4$-skeleta, and $2$-acyclicity of $3$-skeleta --- in case the action is formal, see Example~\ref{exAcyclGKMactions}.

To pursue the above strategy, we prepared a script in Sage available at~\cite{AyzCode}\footnote{the file titled ``Homology\_Of\_Cluster-permutohedra''}. The script was run at a local machine in a single thread. This required about 80Gb RAM, since it involved linear algebra calculations with big matrices. Experiments revealed the following.

\begin{prop}
The following hold for the posets $\Cl_{\Net}$ and $\Cl_{\Sunn}$.
\begin{enumerate}
  \item $\Hr_i(|(\Cl_{\Net})_3|;\Zo)=0$ for $i=0,1,2$;
  \item $\Hr_i(|(\Cl_{\Sunn})_3|;\Zo)=0$ for $i=0,1$
  \item Over $\Zt$ and $\Qo$, Betti numbers of $|(\Cl_{\Net})_4|$ are $\beta_1=\beta_2=0$, while $\beta_3=5$ and $\beta_4=7$.
  \item Over $\Zt$, Betti numbers of $|(\Cl_{\Sunn})_4|$ are equal to $\beta_1=\beta_2=0$, while $\beta_3=5$ and $\beta_4=310$.
\end{enumerate}
\end{prop}

Items 3 and 4 of this proposition show that $|(\Cl_{\Net})_4|$ and $|(\Cl_{\Sunn})_4|$ are not 3-acyclic over $\Zt$, and therefore they are not 3-acyclic over $\Zo$. According to Proposition~\ref{propAcyclicityOfPoset} $M_{\Net,\lambda}$ and $M_{\Sunn,\lambda}$ are not equivariantly formal, at least over $\Zt$ and $\Zo$. This already proves Lemmata~\ref{lemNetGraph} and~\ref{lemSunGraph} for these coefficient rings.

However, in the next part of paper we develop other approaches which independently confirm the result. 

\subsection{ABFP sequence}

In this subsection we review another block of facts known in toric topology. More detailed exposition of some of these facts can be found in~\cite{AyzMasEquiv} and~\cite{AyzMasSolo}.

Let a $k$-dimensional compact torus act on a manifold $X$. Consider the equivariant filtration
\begin{equation}\label{eqXfiltr}
X_0\subset X_1\subset X_2\subset\cdots\subset X_k=X
\end{equation}
where $X_j$ is the union of all orbits of dimension at most $j$. Notice that if the action is noneffective, the filtration stabilizes at $X$ earlier than at $k$-th step. There holds $X_0=X^T$, this is the fixed point set.

\begin{rem}\label{remFiltrationIsUnionFaces}
The filtration term $X_j$ is the union of all invariant submanifolds of rank $j$ in $X$, see Construction~\ref{conInvarFacesEtc}. If $X$ is equivariantly formal, then every invariant submanifold is a face submanifold, see Lemma~\ref{lemInvarIsFormal}. In this case we have $X_j=\bigcup_{\rk F=j}X_F$.
\end{rem}

Filtration~\eqref{eqXfiltr} induces the filtration of the orbit space $Q=X/T$:
\begin{equation}\label{eqQfiltr}
Q_0\subset Q_1\subset Q_2\subset\cdots\subset Q_k=Q,\qquad Q_j=X_j/T.
\end{equation}
If $X$ is equivariantly formal, we have $Q_j=\bigcup_{\rk F=j}F$. The following result was proved by Franz and Puppe in~\cite{FP} in the most general form, however, they refer to Atiyah and Bredon who proved it in equivariant K-theory and rational cohomology respectively.

\begin{prop}[Atiyah--Bredon--Franz--Puppe exact sequence]
Let a $T$-action on $X$ be equivariantly formal and
\begin{itemize}
  \item either $R=\Qo$
  \item or all stabilizers of the action are connected, and $R=\Zo$ or any field.
\end{itemize}
Then there exists a long exact sequence of $H^*(BT;R)$-modules
\begin{multline}\label{eqABseqForX}
0\to H^*_T(X;R)\stackrel{i^*}{\to} H^*_T(X_0;R)\stackrel{\delta_0}{\to}
H^{*+1}_T(X_1,X_0;R)\stackrel{\delta_1}{\to}\cdots\\\cdots
\stackrel{\delta_{k-2}}{\to}H^{*+k-1}_T(X_{k-1},X_{k-2};R)\stackrel{\delta_{k-1}}{\to}H^{*+k}_T(X,X_{k-1};R)\to 0.
\end{multline}
Here the first map is induced by the inclusion $i\colon X_0\hookrightarrow X$, and all other maps $\delta_j$ are the connecting homomorphisms in the long exact sequences of equivariant cohomology of the triples $X_{j-1}\subset X_j\subset X_{j-1}$.
\end{prop}

\begin{rem}
The exactness in the first terms
\[
0\to H^*_T(X;R)\stackrel{i^*}{\to} H^*_T(X_0;R)\stackrel{\delta_0}{\to} H^{*+1}_T(X_1,X_0;R)
\]
for equivariantly formal actions, is the classical result, known as Chang--Skjelbred theorem~\cite{ChSk}.
\end{rem}

One of the important consequences of the Chang--Skjelbred theorem is the ability to compute equivariant cohomology $H^*_T(X;R)$ as the kernel of the homomorphism $\delta_0$. 
The GKM-theory is based on this observation. We review the topological version of GKM-theory which does not assume algebraical torus actions on complex manifolds, as was originally formulated by Goresky--Kottwitz--MacPherson in~\cite{GKM}. The general exposition is compatible with the one given by Kuroki~\cite{Kur}.

\subsection{Topological GKM-theory}

In the following, we assume that all manifolds are orientable.

\begin{defin}\label{definGKMmfd}
A manifold $X$ with $T$-action is called a GKM-manifold (over $R$) if the following holds.
\begin{enumerate}
  \item The action is equivariantly formal (over $R$).
  \item The fixed point set $X^T$ is finite.
  \item The action is $2$-independent, that is, at any fixed point $x\in X^T$, any two tangent weights are non-collinear.
\end{enumerate}
\end{defin}

The third condition implies that the equivariant $1$-skeleton $X_1$ is the union of finitely many invariant 2-spheres $S^2_{pq}$ connecting some pairs $\{p,q\}$ of fixed points. The torus $T$ acts on a sphere $S^2_{pq}$ with some weight $\alpha_{pq}\in \Hom(T,S^1)$. Therefore the structure of the action of $T$ on $X_1$ is encoded by a GKM-graph $G(X)$ which contains the following information:
\begin{itemize}
  \item Vertices of $G(X)$ correspond to fixed points of the action.
  \item For any invariant 2-sphere $S^2_{pq}$ there is an edge $e_{pq}$ between $p$ and $q$.
  \item An edge $e_{pq}$ is labelled with the weight $\alpha_{pq}\in \Hom(T,S^1)\cong H^2(BT;\Zo)$.
\end{itemize}

\begin{rem}
Actually, in the construction of a GKM-graph, we did not use the fact that $X$ is equivariantly formal. So far, with a little abuse of terminology, we can construct GKM-graphs for manifolds satisfying the properties 2-3 in Definition~\ref{definGKMmfd}, even if they are not equivariantly formal. In particular, each isospectral matrix manifold $M_{\Gamma,\lambda}$ satisfies items 2-3, see Proposition~\ref{propGirthAndGraphicahedron}. Therefore the GKM-graph $G(M_{\Gamma,\lambda})$ is well defined. Combinatorially, this graph coincides with the 1-skeleton $(\Cl_\Gamma)_1$ of the cluster-permutohedron. According to Remark~\ref{remTreeClusterGraphic}, the underlying graph of this GKM-graph is nothing but the Cayley graph of $\Sigma_n$ generated by transpositions corresponding to edges of $\Gamma$.
\end{rem}

Chang--Skjelbred theorem asserts that, in equivariantly formal case, the information about equivariant cohomology is contained already in the equivariant 1-skeleton. For GKM-manifolds this implies the following principal result.

\begin{prop}[GKM theorem~\cite{GKM}]\label{propGKMthm}
Let $X$ be a GKM-manifold (over $R$) and $G(X)$ its GKM-graph with the vertex set $\ca{V}$, the edge set $\ca{E}$, and the weights $\alpha=\{\alpha_e\mid e\in \ca{E}\}$. Then there is an isomorphism of $H^*(BT;R)$-algebras:
\[
H_T^*(X;R)\cong \{\phi\colon \ca{V}\to H^*(BT;R)\mid\phi(p)\equiv\phi(q)\mod (\alpha_{e})\mbox{ for any } e\in \ca{E}\},
\]
where the weight $\alpha_{e}$ is considered as an element of $H^2(BT;R)\cong H^2(BT;\Zo)\otimes R$.
\end{prop}

Here, as before, we assume that either all stabilizers of the torus action are connected, and $R$ can be anything, or, in general, $R=\Qo$.

\begin{rem}\label{remEqToOrdinaryRel}
For the ordinary cohomology ring we have the graded ring isomorphism $H^*(X;R)\cong H^*_T(X;R)\otimes_{H^*(BT;R)}R$, according to the equivariant formality of $X$, see Lemma~\ref{lemEquivFormFixedPoints}. Since $H^*(X;R)$ is free over $H^*(BT;R)$, we also have an isomorphism of graded $H^*(BT;R)$-modules
\begin{equation}\label{eqIsoModules}
H^*_T(X;R)\cong H^*(X;R)\otimes_RH^*(BT;R).
\end{equation}
\end{rem}

It is convenient to make computations with Hilbert--Poincare series of graded modules. In most cases, the graded modules are located in even degrees, so we consider series of the form
\[
\Hilb(H^*(X),\sqrt{t})=\sum_{i=0}^{n}\beta_{2i}(X)\cdot t^i, \mbox{ and }
\Hilb(H^*_T(X),\sqrt{t})=\sum_{i=0}^{+\infty}\dim H_T^{2i}(X)\cdot t^i,
\]
neglecting odd degrees. In this notation, isomorphism~\eqref{eqIsoModules} implies
\begin{equation}\label{eqHilbertsRelation}
\Hilb(H^*_T(X),\sqrt{t})=\dfrac{\Hilb(H^*(X),\sqrt{t})}{(1-t)^k},\mbox{ where }k=\dim T.
\end{equation}

\begin{rem}\label{remGenSeries}
If the first $r$ equivariant Betti numbers $\dim H^{2i}_T(X;R)$ are known, then, according to~\eqref{eqHilbertsRelation}, the first $r$ ordinary Betti numbers can be computed by expanding the polynomial
\begin{equation}\label{eqPolynomialsExpansion}
\left(\sum_{i=0}^{r}\dim H^{2i}_T(X;R)\cdot t^i\right)\cdot(1-t)^k=\beta_0+\beta_2t+\beta_4t^2+\cdots+\beta_{2r}t^r+o(t^r).
\end{equation}
\end{rem}

\begin{algor}\label{algorGKM}
Theorem~\ref{propGKMthm} and Remark~\ref{remGenSeries} provide the algorithm to compute first $r$ Betti numbers of a GKM-manifold:
\begin{enumerate}
  \item For each $i=0,1,\ldots,r$ initialize the linear map
  \[
  L_i\colon \bigoplus_{v\in \ca{V}}H^{2i}(BT;R)\to \bigoplus_{e\in\ca{E}}(H^*(BT;R)/(\alpha_e))_{2i},
  \]
  where $(H^*(BT;R)/(\alpha_e))_{2i}$ is the $2i$-th graded component of the quotient algebra. Any homogeneous polynomial $P_v$ of degree $i$ from the summand $H^{2i}(BT;R)$ attached to $v\in\ca{V}$ is mapped to the sum of $[e\colon v]\cdot P_v\mod \alpha_e$ over all edges $e$ incident to $v$. Here $[e\colon v]$ are the incidence signs, defined from arbitrary orientations of edges.
  \item Compute $\dim H_T^{2i}(X;R)=\dim\Ker L_i$.
  \item Compute ordinary Betti numbers $\beta_{2i}(X)$ for $i=0,\ldots,r$ using~\eqref{eqPolynomialsExpansion}.
\end{enumerate}
\end{algor}

\begin{rem}\label{remBaird}
Notice that step 1 in Algorithm~\ref{algorGKM} corresponds to the computation of 0-degree cohomology of a certain sheaf on a GKM graph. The stalks of this sheaf on vertices are the copies of the polynomial algebra, and the stalk on an edge $e$ is the quotient of the polynomial algebra by the ideal $(\alpha_e)$. This sheaf, called the GKM-sheaf, was introduced by Baird in~\cite{Baird}. Since graded components of this sheaf are finite dimensional, the problem of computing Betti numbers can be, in principle, solved algorithmically.
\end{rem}

\begin{rem}\label{remBettiMorse}
The papers~\cite{GZ} and~\cite{BGH} provide an alternative way to compute Betti numbers of GKM-manifolds, and more general abstract GKM-graphs. The technique is based on a combinatorial analogue of Morse theory, this approach is computationally much faster. However, the result that the Morse-type Betti numbers coincide with the Betti numbers computed by Algorithm~\ref{algorGKM}, is proved only for the class of ``inflection-free'' graphs. The 1-skeleta of cluster-permutohedra considered in our paper do not satisfy this condition, so we cannot expect the combinatorial Morse approach to give meaningful results.
\end{rem}

\subsection{GKM-theory and the Net}

Now we are ready to prove that $M_{\Net,\lambda}$ is not equivariantly formal over $\Zt$, $\Qo$, and $\Zo$.

\begin{proof}[Proof of Lemma~\ref{lemNetGraph}]
Assume that $M_{\Net,\lambda}$ is equivariantly formal, so its odd-degree Betti numbers vanish. Then $M_{\Net,\lambda}$ is a GKM-manifold, and its even-degree Betti numbers can be computed by Algorithm~\ref{algorGKM}. We ran a script~\cite{AyzCode}\footnote{the file titled ``GKM\_for\_Cluster-Permutohedra''} to find $\beta_0,\beta_2,\beta_4,\beta_6$ and obtained the result shown in Table~\ref{tableNetGKMbetti}. The Betti numbers are the same over $\Zt$ and over $\Qo$.
\begin{table}[h]
\centering
\begin{tabular}{|c||c|c|c|c|c|c|c|}
\hline
$i$&$0$&$2$&$4$&$6$&$8$&$10$&$12$\\
\hline
$\beta_i$ & $1$ & $20$ & $146$ & $396$ & $146$ & $20$ & $1$ \\
\hline
\end{tabular}
  \caption{Betti numbers of $M_{\Net,\lambda}$ if GKM theorem was applicable}\label{tableNetGKMbetti}
\end{table}
Since $M_{\Net,\lambda}$ is a closed orientable $12$-dimensional manifold, Poincare duality allows to restore the rest Betti numbers. It can be seen that $\sum_{i}\beta_i(M_{\Net,\lambda})=730$. This contradicts to~\eqref{eqTotalBettiNumberFormal}, since $\chi(M_{\Net,\lambda})=6!\neq 730$.
\end{proof}

\begin{rem}\label{remSizesOfMatrices}
For a general graph $\Gamma=([n],E_\Gamma)$, our script computes equivariant Betti numbers of $M_{\Gamma,\lambda}$. We consider noneffective action of $T^n$ on $M_{\Gamma,\lambda}$, since it is easier to generate the matrix $L_i$ in Algorithm~\ref{algorGKM} rather than the matrix corresponding to the effective action of $T^n/\Delta(T^1)$. The matrix $L_i$ has size $A_i\times B_i$ where $A_i={n+i-1\choose i}\cdot n!$ (the number of degree $i$ monomials in $n$ commuting variables times the number of vertices of the graphicahedron), and $B_i={n+i-2\choose i}\cdot \frac{n!|E_\Gamma|}{2}$ (the number of degree $i$ monomials in $n-1$ commuting variables times the number of edges of the graphicahedron).
\end{rem}

In order to obtain a similar contradiction for $M_{\Sunn,\lambda}$ we have to compute Betti numbers up to $\beta_8$ since $\dim_\Ro M_{\Sunn,\lambda}=18$. We could not perform this calculation on the ordinary computer, so we developed another approach leading to contradiction.

\subsection{ABFP sequence and the Sun}

In this subsection we assume that all stabilizers of a $T$-action on a manifold $X$ are connected, so we don't care too much about the coefficient ring. This assumption holds for all manifolds $M_{\Gamma,\lambda}$ according to Proposition~\ref{propGirthAndGraphicahedron}. We formulate a technical statement about the structure of the orbit space of the action, the proof of this statement will be used in the subsequent calculations.

\begin{prop}\label{propInductiveRelative}
Let a $T$-action on $X$ be equivariantly formal. Assume that the following information is known:
\begin{itemize}
  \item The face poset $S(X)$ of the action.
  \item Betti numbers $\beta_i(X_F)$ and ranks of all face submanifolds of $X$, including $X$ itself.
\end{itemize}
Then there exists an algorithm to compute the numbers $\rk H^i(Q,Q_{-1})$ where $Q=X/T$ is the orbit space, and $Q_{-1}$ is the union of all its proper faces.
\end{prop}

\begin{proof}
The proof is by induction on $k=\dim T$, the dimension of the effectively acting torus. The base $k=0$ is trivial, since in this case $Q=X$, $Q_{-1}=\varnothing$, and the Betti numbers of $X$ are known by assumption.

Now consider an arbitrary $k>0$. Since the action is equivariantly formal, we have ABFP sequence~\eqref{eqABseqForX}. This is a long exact sequence of graded vector spaces, therefore taking Euler characteristic in each degree we get
\begin{equation}\label{eqABFPfoHilb}
\Hilb(H_T^*(X);\sqrt{t})=\sum_{j=0}^{k}(-1)^j\Hilb(H_T^{*+j}(X_j, X_{j-1});\sqrt{t}).
\end{equation}
Notice that $X_j$ is the union of all face submanifolds $X_F$ of rank $j$ (see Remark~\ref{remFiltrationIsUnionFaces}), so we have an isomorphism
\[
H_T^*(X_j, X_{j-1})\cong \bigoplus_{\rk F=j} H_T^*(X_F, (X_F)_{-1}),
\]
where $(X_F)_{-1}$ denotes the union of all proper face submanifolds of $X_F$, by the definition it lies in $X_{j-1}$. Furthermore, there is an effective action of $T/T_F$ on $X_F$, which is free\footnote{It would be almost free if we have not required all stabilizers to be connected.} on $X_F\setminus (X_F)_{-1}$. Therefore,
\[
H_T^*(X_F, (X_F)_{-1})\cong H^*(F, F_{-1})\otimes H^*(BT_F).
\]
Summarizing the above isomorphisms, we get
\begin{equation}\label{eqHilbIntermediate}
\Hilb(H_T^{*+j}(X_j, X_{j-1});\sqrt{t})=\sum_{\rk F=j}\dfrac{\Hilb(H^{*+j}(F, F_{-1});\sqrt{t})}{(1-t)^{k-j}}
\end{equation}
Notice that in the case $j=k$ the sum on the r.h.s. consists of a single summand, having trivial polynomial component. This last summand equals $\Hilb(H^{*+k}(Q, Q_{-1});\sqrt{t})$.

Substituting~\eqref{eqHilbIntermediate} and~\eqref{eqHilbertsRelation} into~\eqref{eqABFPfoHilb} we get
\[
\dfrac{\Hilb(H^*(X);\sqrt{t})}{(1-t)^k}=\sum_{j=0}^{k}\dfrac{(-1)^j}{(1-t)^{k-j}}\sum_{\rk F=j}\Hilb(H^{*+j}(F, F_{-1});\sqrt{t}).
\]
Multiplying by $(1-t)^k$ and separating the last term, we obtain
\begin{equation}\label{eqABInter}
\underbrace{\sum_i\beta_i(X)t^i}_{B_X(t)} = \underbrace{\sum_{j=0}^{k-1}\sum_{\rk F=j}\Hilb(H^{*+j}(F, F_{-1});\sqrt{t})(t-1)^j}_{\Inter_X(t)} + \underbrace{\Hilb(H^{*+k}(Q, Q_{-1});\sqrt{t})}_{A_X(t)}(t-1)^k.
\end{equation}
Here the notation $\Inter_X(t)$ stands for ``the intermediate polynomial''. We need to prove that $A_X(t)$ is computable. Notice that
\[
A_X(t)=\dfrac{B_X(t)-\Inter_X(t)}{(t-1)^k},
\]
and $B_X(t)$ is known. For the polynomial $\Inter_X(t)$, there is an expression
\begin{equation}\label{eqInterInductive}
\Inter_X(t)=\sum_{\rk F<k}A_{X_F}(t)(t-1)^j,
\end{equation}
which follows from the definition of all polynomials. The terms on the r.h.s. of \eqref{eqInterInductive} are already computed by induction, since all proper face submanifolds have ranks $<k$. This proves the statement.
\end{proof}

\begin{rem}\label{remOnMasPanovHomologyCells}
Using the inductive argument, described in the proof of Proposition~\ref{propInductiveRelative}, one can prove that, whenever $T^n$ acts on $X^{2n}$ equivariantly formally over $\Zo$ with isolated fixed points, there holds
\[
H^*(F,F_{-1};\Zo)\cong H^*(D^{\rk F},\dd D^{\rk F};\Zo).
\]
Therefore the orbit type filtration of $Q=X/T$ is a homology cell complex, and $Q$ itself is a homology cell. This result was first proved by Masuda and Panov in~\cite{MasPan}. The proof which utilizes ABFP sequence was proposed by the first author in~\cite{AyzMasEquiv}.
\end{rem}

Let $X_h=M_{\Gamma(h),\lambda}$ be the isospectral manifold of staircase matrices determined by a Hessenberg function $h\colon[n]\to[n]$. According to~\cite{AyzStaircase}, $X_h$ is equivariantly formal, and its even degree Betti numbers can be computed through Morse theory. More precisely,
\begin{equation}\label{eqBettiHessenberg}
\Hilb(H^*(X_h);\sqrt{t})=\sum_{\sigma\in\Sigma_n}t^{\inv_h(\sigma)},
\end{equation}
where $\inv_h(\sigma)=\#\{1\leqslant i<j\leqslant h(i)\mid \sigma(i)>\sigma(j)\}$ is the number of inversions in a permutation $\sigma$ subject to Hessenberg function $h$. Since the face poset $S(X_h)$ is known (this is the cluster permutohedron $\Cl_{\Gamma(h)}$), and every face submanifold of $X_h$ is again a manifold $X_{h'}$ for some Hessenberg function $h'$, Proposition~\ref{propInductiveRelative} implies the following

\begin{cor}
For any indifference graph $\Gamma(h)$ there exists an algorithm to compute $\rk H^i(Q,Q_{-1})$ where $Q=M_{\Gamma(h),\lambda}/T$.
\end{cor}

\begin{rem}\label{remWeComputedHessenbergs}
We implemented the algorithm described in the proof of Proposition~\ref{propInductiveRelative} in~\cite{AyzCode}\footnote{the file titled ``Characteristics\_of\_Hess\_Varieties''}. It outputs the list of all indifference graphs with up to 5 vertices, shows the ordinary Betti numbers of the corresponding manifolds $M_{\Gamma(h),\lambda}$ (the polynomials $B_{M_{\Gamma(h),\lambda}}(t)$), and computes relative cohomology of the orbit spaces (the polynomials $A_{M_{\Gamma(h),\lambda}}(t)$).
\end{rem}

\begin{ex}
Our computational result can be checked in the case $\Gamma=K_3$, corresponding to the full flag variety $\Fl_3$. This particular case is well studied. It was proved in~\cite{BT2}, that the orbit space $Q=\Fl_3/T^2$ is homeomorphic to the 4-sphere $S^4$. On the other hand $Q_{-1}$ is just the GKM-graph of the torus action on $\Fl_3$. The latter, as an ordinary graph, is known to be homeomorphic to the complete bipartite graph $K_{3,3}$, see~\cite{GHZ}. Therefore
\[
H^j(Q,Q_{-1};\Zo)\cong H^j(S^4,K_{3,3};\Zo)\cong\begin{cases}
                                          \Zo, & \mbox{if } j=4 \\
                                          \Zo^4, & \mbox{if } j=2 \\
                                          0, & \mbox{otherwise},
                                        \end{cases}
\]
as follows from the long exact sequence of the pair $(S^4,K_{3,3})$. Our script outputs $A(t)=4+t$ for this graph, as expected (notice the degrees' shift in the definition of $A(t)$).
\end{ex}

Finally, we are ready to prove non-equivariant formality of $M_{\Sunn,\lambda}$.

\begin{proof}[Proof of Lemma~\ref{lemSunGraph}]
Within this proof, we write $M$ for $M_{\Sunn,\lambda}$ to simplify notation. Again, we assume that, on the contrary, the effective action of $T^5$ on $M$ is equivariantly formal, and this will lead to contradiction.

Equivariant formality implies that the arguments from the proof of Proposition~\ref{propInductiveRelative} are applicable to $M$, in particular, relation~\eqref{eqABInter} holds true:
\begin{equation}\label{eqInterABforSun}
B_M(t)=\Inter_M(t)+A_M(t)\cdot (t-1)^5.
\end{equation}
The polynomial $\Inter_M(t)$ is expressed as the sum over all proper face submanifolds of~$M$. However, these submanifolds correspond to proper induced subgraphs of $\Sunn$, which are all indifference graphs. Therefore topological characteristics of their isospectral manifolds are computed by induction, see Remark~\ref{remWeComputedHessenbergs}. Gathering all computations together, we get
\[
\Inter_M(t)=306-1362t+2322t^2-1560t^3+540t^4+384t^5+72t^6+18t^7.
\]
Let $B_M(t)=\sum_{j=0}^{9}b_jt^j$ and $A_M(t)=\sum_{j=0}^{4}a_jt^j$, where $a_j$ and $b_j$ are unknown. Formal relation~\eqref{eqInterABforSun} and Poincare duality $b_j=b_{9-j}$, $j=0,\ldots,4$ give the linear system of $14$ equations in $14$ variables $a_j$, $b_j$. Unfortunately, this system is degenerate: there exists a 2-parametric space of solutions, in particular, we have
\begin{equation}\label{eqB0B1B2}
b_0 = -r_1 + 306,\quad b_1 = 5r_1 + r_2 - 1530,\quad b_2 = -10r_1 - 5r_2 + 3120,
\end{equation}
for arbitrary $r_1,r_2\in\Zo$ (other values $b_j$ and $a_j$ are not essential for the arguments). Next, we have $b_0=1$ since $M$ is connected, therefore $r_1=305$, but $r_2$ still remains in the expression.

Now let us compute $\beta_{2j}(M)$ by Algorithm~\ref{algorGKM} using GKM-theory for $j\leqslant 2$, see~\cite{AyzCode}\footnote{the file titled ``GKM\_for\_Cluster-Permutohedra''}. The calculation gives
\begin{equation}\label{eqBeta24Sunn}
\beta_2(M)=5,\quad \beta_4(M)=29
\end{equation}
Putting $b_1=5$ in~\eqref{eqB0B1B2} determines the value $r_2=10$, and we get $b_1=5$, $b_2=20$. This contradicts to $\beta_4(M)=29$ obtained from GKM-theory. This inconsistency shows that the assumption of equivariant formality of $M=M_{\Sunn,\lambda}$ was false.  
\end{proof}

Notice that most arguments in the proof above are purely combinatorial, hence they do not depend on the coefficient field. The calculation of Betti numbers by Algorithm~\ref{algorGKM} outputs the same values of $\beta_2$ and $\beta_4$ for the fields $\Qo$ and $\Zt$.

\section{Real symmetric matrices and graph invariants}\label{secLast}

\subsection{Real symmetric matrices}

In the previous parts of the paper we considered isospectral manifolds of Hermitian complex matrices. We can do the same calculations for their real versions: the manifolds of isospectral real symmetric matrices. The proofs follow the same lines, but some references should be substituted by their discrete torus versions proved recently.

\begin{con}
Let $M_{\Gamma,\lambda}^{\Ro}$ denote the space of all $\Gamma$-shaped real symmetric matrices with the given spectrum $\lambda$. If $\lambda$ is generic, $M_{\Gamma,\lambda}^{\Ro}$ is a smooth closed manifold of real dimension $|E_\Gamma|$. Let $T_\Ro$ denote the discrete group $\{\pm 1\}^n\cong\Zt^n$, we call it a 2-torus or a discrete torus. The group $T_\Ro$ can be identified with the subgroup of $O(n)$ which consists of diagonal matrices with $\pm 1$ on the diagonal. Then $T_\Ro$ acts on symmetric matrices by conjugation: this action preserves both the sparseness type $\Gamma$ and the spectrum. Therefore, we have a smooth $T_\Ro$-action on $M_{\Gamma,\lambda}^{\Ro}$. Notice that the fixed points set $(M_{\Gamma,\lambda}^{\Ro})^{T_\Ro}$ consists of diagonal matrices with $\lambda_i$'s at the diagonal; there are $n!$ isolated fixed points.
\end{con}

We have the following real version of Theorems~\ref{thmMainDtypeChar} and~\ref{thmNotEquivFormal}.

\begin{thm}\label{thmRealShit}
The following are equivalent
\begin{enumerate}
  \item A manifold $M_{\Gamma,\lambda}^{\Ro}$ admits a Morse--Smale system whose stationary points are the diagonal matrices.
  \item The 2-torus action on $M_{\Gamma,\lambda}^{\Ro}$ is equivariantly formal over $\Zt$.
  \item $\Gamma$ is an indifference graph.
\end{enumerate}
\end{thm}

Before proving the theorem we give several important remarks. First, we need to explain what is meant by equivariant formality in the case of discrete torus. The definition of formality can be rewritten, replacing $T$ with $T_\Ro$, and the coefficient ring with the particular field $\Zt$. However, there is an equivalent way to define equivariant formality which is more classical, as well as more convenient in practice. The proof of equivalence of these two approaches can be found e.g. in~\cite[Ch.IV(B) Cor.2]{Hsiang}.

\begin{con}
Let a 2-torus $T_\Ro$ act on a space\footnote{Some assumptions should be imposed on a space, which are certainly satisfied for smooth actions on compact manifolds.} $X$ with $m$ isolated fixed points. Smith theory implies
\begin{equation}\label{eqSmithMain}
m=\dim_{\Zt}H_*(X^{T_\Ro};\Zt)\leqslant \dim_{\Zt}H_*(X;\Zt).
\end{equation}
If there is an equality in~\eqref{eqSmithMain}, the action is called \emph{equivariantly formal} over $\Zt$.
\end{con}

\begin{lem}
Let a 2-torus $T_\Ro$ act on $X$ with $m$ isolated fixed points. If $X$ has a cell structure with $m$ cells, then the 2-torus action is equivariantly formal over $\Zt$.
\end{lem}

\begin{proof}
Cell structure implies
\[
\dim_{\Zt}H_*(X;\Zt)\leqslant \dim_{\Zt}C_*(X;\Zt)=m=\dim_{\Zt} H_*(X^{T_\Ro};\Zt),
\]
which proves the statement.
\end{proof}

Morse theory then implies

\begin{cor}\label{corMorseFormReal}
Let a 2-torus $T_\Ro$ act on $X$ with $m$ isolated fixed points. If there is a Morse--Smale flow on $X$ with $m$ stationary points, then the action is equivariantly formal.
\end{cor}

Now let us prove Theorem~\ref{thmRealShit}.

\begin{proof}[Proof of Theorem~\ref{thmRealShit}]
The 2-torus action on $M_{\Gamma,\lambda}^{\Ro}$ has $n!$ isolated fixed points.

\textbf{(3)$\Rightarrow$(1).} If $\Gamma$ is an indifference graph, then, probably after some relabelling of vertices, we have $\Gamma=\Gamma(h)$ for some Hessenberg function. There is a Toda flow on $M_{\Gamma,\lambda}^{\Ro}$ having $n!$ stationary points. See~\cite{dMP} for details and generalizations to other Lie types.

\textbf{(1)$\Rightarrow$(2).} Apply Corollary~\ref{corMorseFormReal}.

\textbf{(2)$\Rightarrow$(3).} We prove that if $\Gamma$ is not an indifference graph, then $M_{\Gamma,\lambda}^{\Ro}$ is not formal. The complete analogue of Lemma~\ref{lemInvarIsFormal} holds for 2-torus actions, the proof follows from Smith theory. The real analogue of Lemma~\ref{lemInducedGeneral} follows as well: if $\Gamma'$ is an induced subgraph of $\Gamma$, then $M_{\Gamma',\lambda}^{\Ro}$ is contained among invariant submanifolds of $M_{\Gamma,\lambda}^{\Ro}$. Therefore we only need to prove non-formality of $M_{\Gamma,\lambda}^{\Ro}$ for $\Gamma$ being one of the forbidden subgraphs: $\Cy_k, (k\geqslant 4)$, $\St_3$, $\Net$, $\Sunn$.

Notice that the face poset of (the 2-torus action on) $M_{\Gamma,\lambda}^{\Ro}$ is isomorphic to the cluster-permutohedron $\Cl_\Gamma$, the proof is completely similar to its complex version~\cite[Thm.1]{AyzBuchGraph}. All homological arguments used in our proofs can be translated to 2-torus actions, up to division of degrees by $2$. For example, it is a general phenomenon that
\begin{equation}\label{eqRealCompl}
H^{2j}(X;\Zt)\cong H^j(X^{\Ro};\Zt)
\end{equation}
for the real locus $X^{\Ro}$ of a (complex or symplectic) manifold $X$, see~\cite{HHP} for a general exposition of this subject. We give a bit more details and references below.

\begin{enumerate}
  \item[$\St_3$.] In this case the torus action has complexity zero. In the complex case, we utilized the equivariant formality criterion proved by Masuda and Panov~\cite{MasPan}. The real version of this criterion is proved in the recent paper of Yu~\cite{LiYu}. This criterion applies to prove non-formality of $M_{\St_3,\lambda}^{\Ro}$.
  \item[$\Cy_k$.] The acyclicity of skeleta of face posets was proved in~\cite{AyzMasSolo} for torus actions with the proof based on ABFP sequence. The version of ABFP sequence with coefficients in $\Zt$ for actions of discrete 2-tori seem to first be discussed in~\cite{Puppe}. The exactness of this sequence for equivariantly formal actions was proved in~\cite{AFP}. Acyclicity of the skeleta $S(X)_r$ stated in Proposition~\ref{propAcyclicityOfPoset} is proved by specializing ABFP sequence to degree 0 (and applying some technical machinery of homotopy colimits). The same argument works for 2-torus actions: if an action of~$T_\Ro$, $\rk T_\Ro\geqslant 3$ on $X$ is equivariantly formal, then $H_1(|S(X)_2|;\Zt)=0$. Since $H_1(|(\Cl_{\Cy_k})_2|;\Zt)\neq 0$ for $k\geqslant 4$, the manifold $M_{\Cy_k,\lambda}^{\Ro}$ is not formal.
  \item[$\Net$.] In the complex case, we came to contradiction by computing Betti numbers using GKM-theory. The real version of GKM-theory exists as well. One can notice that ``real GKM'' is the consequence of Chang--Skjelbred theorem, which is a part of ABFP sequence, so the fact that ABFP is exact for equivariantly formal 2-torus actions implies the real version of GKM-theory. We also refer to~\cite{Biss} for the related discussion. Since GKM-theory holds true, our computations over $\Zt$ in~\cite{AyzCode} are still valid. Computational experiments show that
      \[
      \dim H_*(M_{\Net,\lambda}^{\Ro};\Zt)=630>6!=\dim H_*((M_{\Net,\lambda}^{\Ro})^{T_\Ro};\Zt)
      \]
      which contradicts to the definition of equivariant formality of a 2-torus action.
  \item[$\Sunn$.] Again, everything follows from the real version of the exact ABFP sequence. It should be noticed that in this case we also used the precise formula~\eqref{eqBettiHessenberg} for Betti numbers of the manifolds of isospectral staircase matrices. The same formula holds for their real loci, up to division of degrees by $2$ and changing coefficient ring to $\Zt$, see~\cite{dMP} and~\cite{AyzStaircase}.
\end{enumerate}
Therefore the whole pipeline of the proof works for discrete torus as well.
\end{proof}

\subsection{Indifference hulls}

The graphs $\Gamma$ which are not indifference graphs produce non-diagonalizable matrix types. However, if $\Gamma\subset\Gamma'$ for an indifference graph $\Gamma'$, then $\Gamma$-shaped matrix can be considered as $\Gamma'$-shaped matrix and can be asymptotically diagonalized in the class of $\Gamma'$-shaped matrices. The motivates the following definition.

\begin{defin}
Let $\adi(\Gamma)$ denote the minimal number of edges needed to be added to $\Gamma$ so that the resulting graph is an indifference graph.
\end{defin}

The number $\adi(\Gamma)$ stores the information on how many additional entries of $\Gamma$-shaped matrix should be stored in memory in one wants to perform an asymptotic diagonalization (e.g. QR-algorithm) on a matrix. 

Since indifference graphs are represented on a line (see Definition~\ref{definIndifGraph}), the problem of computing $\adi(\Gamma)$ is related to the problem of finding the layout of $\Gamma$ on $\Ro$ which is optimal in some sense. This invariant is closely related to the invariant studied in~\cite{KapSham}: the minimal size of the maximal clique among all indifference graphs $\Gamma'$ containing $\Gamma$. The latter invariant is proved to be one greater than \emph{the bandwidth} of a graph. In terms of matrices, the problem of computing the bandwidth corresponds to embedding a $\Gamma$-shaped matrix into a band matrix of the minimal width. Relations between graph algorithms and matrix diagonalization problems are also described with the notion of \emph{treewidth}, see~\cite{FHTr} and references therein.

\begin{ex}
Cycle graphs $\Cy_n$ correspond to periodic tridiagonal matrices,
\begin{equation}\label{eqCyclicMatrix}
\begin{pmatrix}
a_1 & b_1& 0 & \cdots & \overline{b}_n\\
\overline{b}_1& a_2 & b_2 & 0 & \vdots\\
0 & \overline{b}_2 & a_3 & \ddots & 0\\
\vdots & 0 & \ddots &\ddots&  b_{n-1}\\
b_n& \cdots & 0 & \overline{b}_{n-1} &a_n
\end{pmatrix},
\end{equation}
see~\cite{AyzPeriodic} for details. We have $\adi(\Cy_n)=n-3$. Indeed, a particular way of turning a cycle into an indifference graph is shown on Fig.~\ref{figCycleSew}. The optimality of such pattern can be easily proven by induction on $n$. This means that a periodic tridiagonal matrix can be embedded into a pentadiagonal matrix: the one which corresponds to the Hessenberg function $(3,4,5,\ldots,n,n)$, or the graph shown on the right part of Fig.~\ref{figCycleSew}. This may seem counterintuitive at first glance, since the obvious way of making~\eqref{eqCyclicMatrix} into a Hessenberg matrix is to fill out the whole matrix. However, one should remember that reordering of rows and columns is allowed, which makes the described trick possible.
\end{ex}

\begin{figure}[h]
\begin{center}
\includegraphics[scale=0.35]{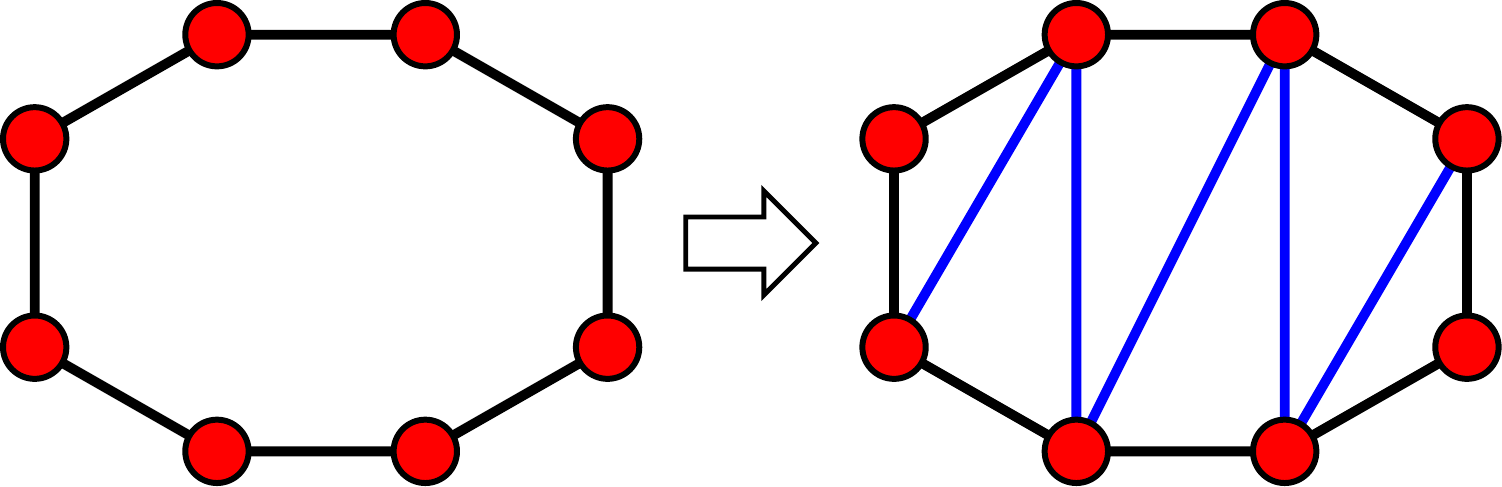}
\end{center}
\caption{Embedding of a cycle into an indifferent graph}\label{figCycleSew}
\end{figure}

\begin{rem}
The previous example implies that for any graph $\Gamma$, there holds
\begin{equation}\label{eqGirthSeam}
\adi(\Gamma)\geqslant\girth(\Gamma)-3.
\end{equation}
For certain, there exists a number of relations of $\adi(\cdot)$ to other known graph invariants. These relations, as well as computational complexity issues will be addressed in a different paper.
\end{rem}

\section*{Acknowledgements}
We thank Oleg Kachan and Eduard Tulchinskiy for their persistent help with parallelizing some of the computations at HSE University supercomputer ``cHARISMa''. Li Yu and Vlad Gorchakov had shared some relevant references on 2-torus actions which were quite helpful. The first author thanks Prof. Mikiya Masuda for organizing the workshop ``Hessenberg varieties in Osaka 2019'' where a very fruitful discussion of graph-theoretical approaches to Stanley--Stembridge conjecture had emerged that gave a push to this study. The conference on data science in Voronovo organized by Evgeny Sokolov, motivated the first author to find connections between toric topology and gradient descent algorithms which eventually led to this research.

\end{document}